\documentclass[11pt]{article}

\usepackage[T1]{fontenc}
\usepackage{lmodern}
\usepackage{amsmath,amssymb,amsthm,mathtools}
\usepackage[margin=1.08in]{geometry}
\usepackage{microtype}
\usepackage{booktabs}
\usepackage{xcolor}
\usepackage{hyperref}
\usepackage{fancyhdr}

\hypersetup{
  colorlinks=true,
  linkcolor=blue!55!black,
  citecolor=blue!55!black,
  urlcolor=blue!55!black,
  pdftitle={A Nonmonotone Real-Rootedness Set for Symmetric Imaginary Shifts},
  pdfauthor={Vasily Stodolsky; Independent Researcher; ORCID 0009-0003-5329-5308},
  pdfsubject={AI research artifact; model-generated mathematical content},
  pdfkeywords={real-rooted polynomials, hyperbolic polynomials, imaginary shifts, Sturm theorem, finite-difference operators}
}

\newtheorem{theorem}{Theorem}[section]

\newcommand{\OmegaF}{\Omega_F}

\title{\textbf{A Nonmonotone Real-Rootedness Set for\\
Symmetric Imaginary Shifts}\\[0.7em]
\large\textnormal{AI Research Artifact}}
\author{Vasily Stodolsky\\
\small Independent Researcher\\
\small \href{https://orcid.org/0009-0003-5329-5308}
{ORCID 0009-0003-5329-5308}}
\date{Version 0.1.3, August 13, 2026\\Authorship and provenance update}

\begin{document}

\maketitle

\begin{center}
\begin{minipage}{0.89\textwidth}
\small
\textbf{Authorship and accountability.}
The theorem statements and candidate proofs were materially generated and
revised with AI systems. Vasily Stodolsky is the named author, approved this
version, and accepts responsibility for all its contents. The end-matter AI Use
and Provenance Statement records the concrete human and model contributions.
No independent human peer review is claimed.

\smallskip
\textbf{Version note.}
Version 0.1.3 replaces mental-state and submission-decision wording with an
action-based authorship and accountability statement and aligns Zenodo and
arXiv metadata. The theorem statement, proof arguments, polynomial, and exact
certificates are unchanged from version 0.1.2.
\end{minipage}
\end{center}

\begin{abstract}
For a real polynomial $F$ and $\omega\geq 0$, let
\[
  A_\omega(z)=\frac{F(z+i\omega)+F(z-i\omega)}{2},
  \qquad
  \OmegaF=\{\omega\geq0:A_\omega\text{ has only real zeros}\}.
\]
This artifact presents an explicit rational even polynomial of degree eight for which
$6/25$ and $12/25$ belong to $\OmegaF$, while $3/10$ does not. Exact Sturm
certificates give respectively eight, four, and eight distinct real zeros.
Consequently $\OmegaF$ is neither an interval nor an up-set. All zeros of
$F$ lie in the strip $|\operatorname{Im}z|\leq11/25$, and the classical
strip-contraction theorem gives the eventual tail
$[11/25,\infty)\subset\OmegaF$. We also include a direct elementary proof
of that tail and a standard-library exact verifier.
\end{abstract}



\medskip
\noindent\textbf{Persistent record.}
The version history is maintained at
\href{https://doi.org/10.5281/zenodo.21728324}
{10.5281/\allowbreak zenodo.21728324}.

\section{Introduction}

Let $D$ denote differentiation. For a polynomial $F$, the operator
\[
  \cos(\omega D)F(z)
  =\frac{F(z+i\omega)+F(z-i\omega)}{2}
\]
is a two-term finite-difference operator with imaginary shifts. Classical
results of de Bruijn localize the zeros of such combinations and give
quantitative contraction of a horizontal zero strip
\cite{deBruijn1950}. Br\"and\'en and Chasse give a precise modern
strip-preserver formulation: if all zeros of a real polynomial lie in
$|\operatorname{Im}z|\leq\mu$, then the output strip has half-width
\[
  \sqrt{\max\{\mu^2-\omega^2,0\}};
\]
see \cite[Theorem 1.6]{BrandenChasse2017}. In particular, every
$\omega\geq\mu$ gives a real-rooted output.

The universal theory of hyperbolicity and stability preservers is much
broader; see Borcea and Br\"and\'en \cite{BorceaBranden2009}. Related
finite-difference classifications and large-step asymptotics appear in
Katkova, Tyaglov, and Vishnyakova \cite{KTV2020}. Adams and Cardon study a
symmetric imaginary-shift sum when the input already has only real zeros
\cite{AdamsCardon2007}, while Cardon studies universal complex zero strip
decrease for differential operators, including cosine symbols
\cite{Cardon2015}. Nuij's deformation $p+s p'$ concerns paths and topology
inside spaces of hyperbolic polynomials \cite{Nuij1968}. The classical
multiplier-sequence framework of P\'olya and Schur concerns diagonal
coefficient operators \cite{PolyaSchur1914}; it does not determine the
fixed-input set $\OmegaF$ considered here.

These results control whole input classes or the eventual strip tail. A
different question is what the parameter set
\[
  \OmegaF=\{\omega\geq0:\cos(\omega D)F\text{ is real-rooted}\}
\]
can look like for one fixed polynomial whose initial zeros are not all real.
The example below shows that the set need not be monotone and need not be an
interval, even when the input polynomial is real and even, its zeros lie in a bounded strip,
and an eventual real-rooted tail is present.

The claim is deliberately narrow. No assertion of historical priority or
minimal degree is made. A bounded source audit found no exact anticipation
or stronger containment of the displayed finite witness. The complete
1914 P\'olya-Schur article was inspected as background. One older monograph
remains available only through an authoritative bibliographic record and
modern formulations. The accompanying audit summary records that limit.

\section{The polynomial and the main result}

Define
\begin{align}
F(z)={}&
\left(\left(z-\frac1{50}\right)^2+\left(\frac7{25}\right)^2\right)
\left(\left(z+\frac1{50}\right)^2+\left(\frac7{25}\right)^2\right)
\notag\\
&\quad\times
\left(\left(z-\frac35\right)^2+\left(\frac{11}{25}\right)^2\right)
\left(\left(z+\frac35\right)^2+\left(\frac{11}{25}\right)^2\right).
\label{eq:F}
\end{align}
For $\omega\geq0$, set
\begin{equation}
  A_\omega(z)=\frac{F(z+i\omega)+F(z-i\omega)}{2},
  \qquad
  \OmegaF=\{\omega\geq0:A_\omega\text{ has only real zeros}\}.
  \label{eq:Aomega}
\end{equation}

\begin{theorem}\label{thm:main}
The polynomial $F$ in \eqref{eq:F} is real and even, and every zero of $F$
lies in $|\operatorname{Im}z|\leq11/25$. Moreover,
\[
  \frac6{25}\in\OmegaF,
  \qquad
  \frac3{10}\notin\OmegaF,
  \qquad
  \frac{12}{25}\in\OmegaF.
\]
The respective numbers of distinct real zeros of $A_\omega$ are $8$, $4$,
and $8$. In addition,
\[
  \left[\frac{11}{25},\infty\right)\subset\OmegaF.
\]
Consequently $\OmegaF$ is neither an up-set nor an interval.
\end{theorem}

\begin{proof}[Proof of the symmetry and strip assertions]
The zeros of $F$ are visible directly from \eqref{eq:F}:
\[
 \frac1{50}\mathbin{\pm}\frac7{25}i,
 -\frac1{50}\mathbin{\pm}\frac7{25}i,
 \frac35\mathbin{\pm}\frac{11}{25}i,
 -\frac35\mathbin{\pm}\frac{11}{25}i.
\]
This proves the symmetry and strip assertions. The two remaining parts of
the theorem are proved separately below.
\end{proof}

\section{Proof idea and verification route}

This section is an explanatory map to the formal certificate below. It does
not add a theorem, replace the exact proof, or make a claim about the complete
shape of $\OmegaF$. Its purpose is to make clear why one finite polynomial,
three rational parameters, and exact Sturm computations suffice for the stated
separation result.

\subsection*{Motivation and mechanism}

The point of the three parameters is their order:
\[
  \frac6{25}=0.24<\frac3{10}=0.30<\frac{12}{25}=0.48.
\]
At the first and third values, $A_\omega$ has all eight zeros real, while
at the middle value it has only four real zeros and hence four nonreal zeros.
An interval containing both $0.24$ and $0.48$ must contain the intermediate
point $0.30$. Likewise, an up-set containing $0.24$ must contain every larger
parameter, including $0.30$. The in-out-in pattern therefore rules out both
an interval and a monotone threshold of good parameters. It is the pattern in
one continuously parametrized family, not the individual root counts in
isolation, that provides the counterexample. In particular, real-rootedness
at one parameter and throughout an eventual tail cannot be joined by a purely
structural monotonicity argument; additional information about the family is
required. This is not presented as a refutation of a named conjecture, but as
a separation result showing that the stated structural assumptions do not
imply monotonicity.

The mechanism is deliberately finite. The polynomial is real and even, with
four conjugate pairs at two rational horizontal locations and two rational
strip heights. Symmetric imaginary translation preserves evenness. Thus the
degree-eight question for $A_\omega$ can be re-expressed as a degree-four
question in $t=z^2$. This reduction does not approximate the roots: it records
the symmetry already present in the displayed polynomial. Positive roots in
$t$ correspond in pairs to real roots in $z$, while a failure to obtain four
positive roots prevents all eight roots of $A_\omega$ from being real.

\subsection*{Role of the assumptions and logical route}

Reality ensures that the shift average has real coefficients, so Sturm's
theorem can count its real roots exactly. Evenness supplies the substitution
$t=z^2$, reducing the certificate to quartics without changing its conclusion.
The visible factor form of $F$ also locates its zeros in a strip of half-width
$11/25$. That strip information is used only for the eventual-tail statement.
It is not a substitute for the three finite certificates, and it does not
imply monotonicity at smaller parameters.

The formal route is directional. First, the factorization establishes the
symmetry and strip assertions. Second, exact expansion gives one quartic
$P_\omega$ at each selected rational parameter. Third, a Sturm chain
determines the number of positive roots of each quartic from endpoint sign
variations. Fourth, the substitution $t=z^2$ converts those counts to the
real-root counts of the degree-eight averages. Finally, the ordered in-out-in
pattern yields the two set-theoretic consequences. The later tail argument is
independent supporting information: it explains why an eventual real-rooted
regime can coexist with the finite nonmonotone pattern.

\subsection*{Sturm certificates and the exact verifier}

The exact certificate is shorter after using evenness. We write
$A_\omega(z)=P_\omega(z^2)$, where $P_\omega$ is a real quartic. Sturm's
theorem counts four, two, and four positive roots of these quartics. Every
positive root $t$ produces the two real roots $z=\mathbin{\pm}\sqrt t$, so
the corresponding degree-eight polynomials have eight, four, and eight real
zeros. The expanded rational coefficients only keep this sign computation
exact; the sign-variation table is the actual root-count certificate. The
separate tail argument shows that all $\omega\geq11/25$ are good, so the
eventual classical regime coexists with nonmonotone behavior below it. The
certificate is exact throughout: coefficients, Euclidean remainders, endpoint
signs, and variation counts are rational computations. The accompanying
standard-library verifier recomputes these objects using exact fractions. It
is a reproducibility check for the displayed certificate, not a replacement
for mathematical reading and not a claim of human peer review.

\subsection*{What this section does not claim}

The argument establishes the displayed membership and non-membership facts,
the corresponding root counts, and the stated eventual tail. It does not
classify every parameter in $\OmegaF$, prove that degree eight is minimal,
identify a first historical example, or decide monotonicity for a special
family with additional structure. No numerical root approximation is used as
the certificate. The AI-assisted origin and model-audit history are disclosed
separately; they are provenance information and are not presented as an
independent human referee report.

\section{Exact Sturm certificate}

\begin{proof}[Proof of the membership claims]
Because $F$ is even, every $A_\omega$ is even. Write
\[
  A_\omega(z)=P_\omega(z^2),
\]
where $P_\omega$ is a real quartic. Exact expansion gives
\begin{align*}
P_{6/25}(t)
={}&t^4-\frac{2237}{1250}t^3+\frac{161441}{250000}t^2
-\frac{5715517}{97656250}t+\frac{173989}{976562500},\\
P_{3/10}(t)
={}&t^4-\frac{3371}{1250}t^3+\frac{1333057}{1250000}t^2
-\frac{267506531}{1953125000}t
+\frac{906531301}{9765625000000},\\
P_{12/25}(t)
={}&t^4-\frac{1657}{250}t^3+\frac{5734597}{1250000}t^2
-\frac{194818069}{244140625}t
+\frac{1243480589}{122070312500}.
\end{align*}

For any one of these quartics, define its Sturm chain by
\[
  S_0=P_\omega,
  \qquad S_1=P_\omega',
  \qquad S_{j+1}=-\operatorname{rem}(S_{j-1},S_j).
\]
Exact Euclidean division over $\mathbb{Q}$ gives degrees
$4,3,2,1,0$ in all three cases. The endpoint signs and variation counts are
as follows. No entry in a displayed sign row is zero.

\begin{center}
\small
\begin{tabular}{@{}c c c c c c@{}}
\toprule
$\omega$ & $\operatorname{sgn}S(0^+)$ & $V(0^+)$
& $\operatorname{sgn}S(+\infty)$ & $V(+\infty)$
& positive roots of $P_\omega$ \\
\midrule
$6/25$  & $+\ -\ +\ -\ +$ & $4$ & $+\ +\ +\ +\ +$ & $0$ & $4$ \\
$3/10$  & $+\ -\ +\ -\ -$ & $3$ & $+\ +\ +\ +\ -$ & $1$ & $2$ \\
$12/25$ & $+\ -\ +\ -\ +$ & $4$ & $+\ +\ +\ +\ +$ & $0$ & $4$ \\
\bottomrule
\end{tabular}
\end{center}

The last member of each chain is a nonzero constant, so each quartic is
square-free. Also $P_\omega(0)>0$. Sturm's theorem therefore gives exactly
four, two, and four distinct positive roots of the three quartics. Under
$t=z^2$, each positive root yields the two distinct real roots
$z=\mathbin{\pm}\sqrt{t}$. Thus the three polynomials $A_\omega$ have
respectively eight, four, and eight distinct real roots. Since each
$A_\omega$ has degree eight, the middle polynomial is not real-rooted. This
proves the three membership claims.
\end{proof}

The exact verifier included with the release reconstructs $F$ from its four
quadratic factors, expands all three vertical averages over $\mathbb{Q}$,
computes both the quartic chains above and the degree-eight Sturm chains, and
checks square-freeness and the root counts. It uses only the Python standard
library.

\section{The classical real-rooted tail}

The containment
\[
  [11/25,\infty)\subset\OmegaF
\]
is a special case of the de Bruijn strip-contraction theorem, in the precise
form of \cite[Theorem 1.6]{BrandenChasse2017}. We include an elementary proof
for this finite polynomial.

\begin{proof}[Proof of the tail]
Fix $\omega\geq11/25$ and let $z=x+iy$ with $y>0$. For one conjugate pair
$\beta\mathbin{\pm}i\gamma$ of zeros of $F$, put
\[
  a=(x-\beta)^2,
  \qquad
  g(s)=\bigl(a+(s-\gamma)^2\bigr)
       \bigl(a+(s+\gamma)^2\bigr).
\]
The quantity $g(s)$ is the squared modulus contributed by this pair at
imaginary coordinate $s$. Set $p=y-\omega$ and $q=y+\omega$. Direct
subtraction gives
\begin{align*}
  g(q)-g(p)
  &=8y\omega\bigl(y^2+\omega^2+a-\gamma^2\bigr)\\
  &>0.
\end{align*}
Indeed, every pair in \eqref{eq:F} satisfies
$|\gamma|\leq11/25\leq\omega$, and $y>0$. Multiplying the strict inequality
over the four conjugate pairs yields
\[
  |F(z+i\omega)|>|F(z-i\omega)|.
\]
The two summands in \eqref{eq:Aomega} therefore cannot cancel in the upper
half-plane. Since $A_\omega$ is a real polynomial, conjugation excludes zeros
in the lower half-plane as well. Hence every zero of $A_\omega$ is real.
\end{proof}

\section{Consequences and scope}

Recall that a subset $U\subset[0,\infty)$ is an up-set if
$\omega\in U$ and $\sigma>\omega$ imply $\sigma\in U$. Here
\[
  \frac6{25}<\frac3{10}<\frac{12}{25}.
\]
Because $6/25\in\OmegaF$ but $3/10\notin\OmegaF$, the set $\OmegaF$ is not
an up-set. Because both endpoints $6/25$ and $12/25$ belong to $\OmegaF$
while the intermediate point $3/10$ does not, $\OmegaF$ is not an interval.

The example refutes only a generic structural implication: real and even
input, conjugation-symmetric strip zeros, and an eventual real-rooted tail do
not force monotonicity of the fixed-input real-rootedness set. The result does
not characterize all of $\OmegaF$, prove a lower bound on the degree of any
such example, or address a distinguished transcendental or arithmetic
family. Additional structure in a special family could still impose a
monotonicity theorem absent from the finite example.

\section{Reproducibility, AI use, and provenance}

The release package contains the LaTeX source, this PDF, the standard-library
exact verifier, a recorded verifier output, an audit summary, a bounded
prior-art summary, and a SHA-256 manifest. Running
\[
  \texttt{python -B tools/verify\_d05\_sturm.py}
\]
reproduces the coefficient maps, Sturm signs, variation counts, real-root
counts, and square-freeness checks.

The explicit witness and candidate proof arose in a GPT-5.5/xhigh headless
quota-failover continuation.
Claude Opus 5 Max later audited and reconciled the source record; this was not
the initial discovery. Fresh GPT-5.5/xhigh sessions were then used for
independent proof reconstruction, repair verification, and post-audit
cross-checking. GPT-5.6-sol/max was used for the focused prior-art
adjudication. The exact arithmetic was also recomputed in fresh sessions and
by the released deterministic script. Model cross-checking does not establish
human peer review or historical priority. The accompanying workflow file
gives the detailed roles and limitations.

Vasily Stodolsky directed the research workflow, set its scope and release
criteria, curated the mathematical and verification record, made the final
decisions on claims and revisions, and approved this version. He accepts
responsibility for all contents of the work. The AI systems are research tools
and are not listed as authors. No independent human peer review is claimed.
Vasily Stodolsky is identified by ORCID 0009-0003-5329-5308.

\medskip
\noindent\textbf{License.}
The manuscript and documentation are licensed under CC BY 4.0. Executable
code in the release is licensed under the MIT License.

\footnotesize
\bibliographystyle{abbrv}
\bibliography{references}

\end{document}